\documentclass[amscd,verbatim,12pt]{amsart}

\usepackage{tikz}
\usepackage{amssymb,amsmath,latexsym,graphics,youngtab,supertabular}
\usepackage{amsbsy}
\usepackage{amscd}
\usepackage{amsfonts}
\usepackage{amsthm}
\usepackage{epsfig}
\usepackage{times}
\usepackage{mathrsfs}
\usepackage{verbatim}
\usepackage{url}
\usepackage{graphicx}
\usepackage{txfonts}
\usepackage{hyperref} 

\usepackage[english]{babel}
\usepackage[T1]{fontenc}
\usepackage[utf8]{inputenc}

\usetikzlibrary{backgrounds,fit,decorations.pathreplacing}
\usetikzlibrary{arrows,shapes,positioning}
\usetikzlibrary{calc,through,backgrounds,chains}

\usetikzlibrary{arrows,shapes,snakes,automata,backgrounds,petri}

\usepackage[version=3]{mhchem}

\newtheorem{theorem}{Theorem}[section]

\theoremstyle{definition}

\newtheorem{remark}[theorem]{Remark}

\begin{document}

\title{Proofs Of Three Geode Conjectures }


\author[T. Amdeberhan]{Tewodros Amdeberhan}
\address{Department of Mathematics,
Tulane University, New Orleans, LA 70118, USA}
\email{tamdeber@tulane.edu}
\author[D. Zeilberger]{Doron Zeilberger}
\address{Rutgers University, Department of Mathematics, 110 Frelinghuysen Rd, Piscataway, NJ 08854, USA}
\email{DoronZeil@gmail.com}

\begin{abstract} In the May 2025 issue of the Amer. Math. Monthly, 
Norman J. Wildberger and Dean Rubine introduced a new kind of multi-indexed numbers, that they call `Geode numbers', obtained from the Hyper-Catalan numbers. They posed
three intriguing conjectures about them, that are proved in this note.

\end{abstract}

\subjclass[2000]{Primary, Secondary}

\date{\today}

\maketitle

\newcommand{\nn}{\nonumber}
\newcommand{\ba}{\begin{eqnarray}}
\newcommand{\ea}{\end{eqnarray}}
\newcommand{\realpart}{\mathop{\rm Re}\nolimits}
\newcommand{\imagpart}{\mathop{\rm Im}\nolimits}
\newcommand{\lcm}{\operatorname*{lcm}}
\newcommand{\vGam}{\varGamma}

\newtheorem{Definition}{\bf Definition}[section]
\newtheorem{Thm}[Definition]{\bf Theorem}
\newtheorem{Theorem}[Definition]{\bf Theorem}
\newtheorem{Example}[Definition]{\bf Example}
\newtheorem{Lem}[Definition]{\bf Lemma}
\newtheorem{Rm}[Definition]{\bf Remark}
\newtheorem{Note}[Definition]{\bf Note}
\newtheorem{Cor}[Definition]{\bf Corollary}
\newtheorem{Prop}[Definition]{\bf Proposition}
\newtheorem{Conj}[Definition]{\bf Conjecture}
\newtheorem{Problem}[Definition]{\bf Problem}
\numberwithin{equation}{section}

\section{Introduction}

\noindent
In a recent captivating Monthly article \cite{WR},
by Norman J. Wildberger and Dean Rubine, the authors utilize a generating series to solve the general univariate polynomial equation. 
They also explored a ``curious factorization" of this hyper-Catalan generating series,
and in the penultimate section, they made three conjectures about this algebraic object that they termed the \emph{Geode array}. 

\smallskip
\noindent
In this note, we prove these three conjectures. At least as interesting as the actual statements of the conjectures (now theorems)
is {\it how we proved them}, using several important {\it tools of the trade}.

The first tool is the \emph{multinomial theorem}
\begin{align} \label{multi}
(x_1+\cdots+x_r)^n & = \sum_{\substack{m_1,\dots.m_r\geq0 \\ m_1+\cdots+m_r=n}} \binom{n}{m_1,\dots,m_r} x_1^{m_1}\cdots x_r^{m_r}.
\end{align}

\noindent
The second tool is \emph{constant-term extraction},  the third is 
\emph{Wilf-Zeilberger (WZ) algorithmic proof theory}\cite{WZ} and the last-but-not-least tool is
\emph{Lagrange Inversion} \cite{Z}  that states that: \emph{if $u(t)$ and $\Phi(t)$ are formal power series starting at $t^1$ and $t^0$,
respectively, then $u(t) = t\Phi(u(t))$ implies
\begin{align} \label{lag} 
[t^n] u(t) & =\frac1n [z^{n-1}] \Phi(z)^n.
\end{align} 
}
Here $[z^n] F(z)$ means the coefficient of $z^n$ in the Laurent expansion of $F(z)$.
We shall use the notation $\pmb{CT}_z F(z)$ for the constant-term of $F(z)$.

\smallskip
\noindent
We now bring in the relevant notation adopted in \cite{WR} with a caveat that indices are shifted slightly. Consider the equation $0=1-\alpha+\sum_{k\geq1} t_k\alpha^{k+1}$ and denote its series solution by $\alpha=\pmb{S}[t_1, t_2, \dots]$. Letting $\pmb{S}_1=t_1+t_2+\cdots$, Wildberger-Rubine proved \cite[Theorem 12]{WR} the existence of a (remarkable!) factorization $\pmb{S}-1=\pmb{S}_1\pmb{G}$ and the factor $\pmb{G}[t_1,t_2,\dots]$ 
(that they dubbed the \emph{Geode series}). 
Furthermore, we opt to use $G[m_1,m_2,\dots]$ for the coefficient of $t_1^{m_1}t_2^{m_2}\cdots$ in the polyseries $\pmb{G}[t_1,t_2,\dots]$. 
We are now ready to state and prove the three  conjectures from \cite[p. 399]{WR}. For the sake of clarity, let's describe the first of these  in some detail.

\smallskip
\noindent
Suppose we are solving the polynomial equation $0=1-\alpha+t_1\alpha^2+t_2\alpha^3$ through the formal power series 
$$\alpha=\pmb{S}[t_1,t_2]=\sum_{m_1,m_2\geq0} C[m_1,m_2]\,t_1^{m_1}t_2^{m_2}.$$
Consequently, the corresponding Geode series becomes $\pmb{G}[t_1,t_2]=\frac{\pmb{S}[t_1,t_2]-1}{t_1+t_2}$.
We follow closely \cite{Z} 
to engage the Lagrange Inversion in the extraction of the coefficients $C[m_1,m_2]$ satisfying $n=m_1+m_2$. Then, the amalgamation of such monomials is given by \eqref{lag} in the form of
\begin{align*}
\sum_{m_1+m_2=n} C[m_1,m_2]\, t_1^{m_1}t_2^{m_2}
& =[Y^n]\left(\sum_{k=1}^{3n+1}\frac1k\, [z^{k-1}]\left(1+Yt_1z^2+Yt_2z^3\right)^k\right)  \\
& =[Y^n] \sum_{m_1,m_2\geq0} \frac{\binom{1+2m_1+3m_2}{m_1,m_2,1+m_1+2m_2} }{1+2m_1+3m_2} Y^{m_1+m_2}t_1^{m_1}t_2^{m_2}  \\
& = \sum_{\substack{m_1,m_2\geq0 \\ m_1+m_2=n}} \frac{\binom{1+2m_1+3m_2}{m_1,m_2,1+m_1+2m_2} }{1+2m_1+3m_2} t_1^{m_1}t_2^{m_2}   \\
& = \sum_{m_2=0}^n \frac{\binom{1+2n+m_2}{n-m_2,m_2,1+n+m_2} }{1+2n+m_2} t_1^{n-m_2}t_2^{m_2}   \\
& = \sum_{k=0}^n \frac{\binom{n}k\binom{2n+1+k}{n+1+k}}{2n+1+k} t_1^{n-k}t_2^k.
\end{align*}
For example, the following reveal both coefficients $C[m_1,m_2]$ and $G[m_1,m_2]$:
\begin{align*}
\sum_{m_1+m_2=3} C[m_1,m_2]\, t_1^{m_1}t_2^{m_2}
&=(t_1+t_2)(5t_1^2 + 16t_1t_2 + 12t_2^2),  \\
\sum_{m_1+m_2=4} C[m_1,m_2]\, t_1^{m_1}t_2^{m_2}
& = (t_1+t_2)(14t_1^3 + 70t_1^2t_2 + 110t_1t_2^2 + 55t_2^3).
\end{align*}
As a first step, we reprove that the linear term $t_2+t_3$ divides the polynomial
$$P_n(t_1,t_2):= \sum_{k=0}^n \frac{\binom{n}k\binom{2n+1+k}{n+1+k}}{2n+1+k} t_1^{n-k}t_2^k.$$
This is equivalent to proving that $P_n(-t_2,t_2)=0$, which,  in turn, is equivalent to the following identity:
$$ \sum_{k=0}^n (-1)^k\frac{\binom{n}k\binom{2n+1+k}{n+1+k}}{2n+1+k} =0.$$
To continue, we invoke the role of the WZ method. 
Define the functions $F(n,k):= (-1)^k\frac{\binom{n}k\binom{2n+1+k}{n+1+k}}{2n+1+k}$ and also $H(n,k):=-F(n,k)\cdot \frac{k(n+1+k)}{n(2n+1)}$ to verify $F(n,k)=H(n,k+1)-H(n,k)$. The rest is routine \cite{WZ}.

\smallskip
\noindent
Our next step will actually find $G[m_1,m_2]$. For that we perform the division $\frac{P_n(t_1,t_2)}{t_1+t_2}$
to obtain (algebraically) that
\begin{align*}
[t_1^{n-1-i}t_2^i]\left(\frac{P_n(t_1,t_2)}{t_1+t_2}\right)
& =\sum_{j=0}^i (-1)^{i-j}\frac{\binom{n}j\binom{2n+1+j}{n+1+j}}{2n+1+j}  \\
& = (-1)^i[H(n,i+1)-H(n,0)]  \\
& = (-1)^i H(n,i+1) \\
& = \frac1{2n+1}\binom{n-1}i \binom{2n+1+i}{n+1+i}
\end{align*}
which leads to (an equivalent form of) the first conjecture \cite{WR} on $G[m_1,m_2]$. To wit:

\begin{theorem} \label{conj1} For non-negative integers $m_1$ and $m_2$, we have
$$G[m_1,m_2]=\frac1{(2m_1+2m_2+3)(m_1+m_2+1)} \frac{(2m_1+3m_2+3)!}{(m_1+2m_2+2)!m_1!m_2!}.$$
\end{theorem}

\section{On the second conjecture}

\noindent
Now that the reader, hopefully, is getting accustomed to our proof-procedure as depicted in Section 1, let's move on to next conjecture \cite[p. 399]{WR} which does generalize the one we just finished proving.
For brevity, denote $\widetilde{G}=\widetilde{G}[m_a,m_{a+1}]=G[0,0,\dots,m_a,m_{a+1}]$. Again, we revive the Lagrange Inversion \eqref{lag}. Suppose $n=m_a+m_{a+1}$. Then the total content of such monomials is encapsulated by
\begin{align*}
\sum_{m_a+m_{a+1}=n} \widetilde{G}\, t_a^{m_2}t_{a+1}^{m_3}
& =\frac{[Y^n]}{t_a+t_{a+1}} \sum_{k=1}^{(a+1)n+1}\frac1k\, [z^{k-1}]\left(1+Yt_az^a+Yt_{a+1}z^{a+1}\right)^k  \\
& = \frac{[Y^n]}{t_a+t_{a+1}} \sum_{m_a,m_{a+1}\geq0} \frac{\binom{1+am_a+(a+1)m_{a+1}}{m_a,m_{a+1},1+(a-1)m_a+am_{a+1}} Y^{m_a+m_{a+1}}t_a^{m_a}t_{a+1}^{m_{a+1}} }{1+am_a+(a+1)m_{a+1}}  \\
& = \sum_{\substack{m_a,m_{a+1}\geq0 \\ m_a+m_{a+1}=n}} \frac{\binom{1+am_a+(a+1)m_{a+1}}{m_a,m_{a+1},1+(a-1)m_a+am_{a+1}}  }{1+am_a+(a+1)m_{a+1}}
\frac{t_a^{m_a}t_{a+1}^{m_{a+1}} }{t_a+t_{a+1}}  \\
& = \sum_{m_{a+1}=0}^n \frac{\binom{1+an+m_{a+1}}{n-m_{a+1},m_{a+1},1+(a-1)n+m_3} }{1+an+m_{a+1}} 
\frac{t_a^{n-m_{a+1}} t_{a+1}^{m_{a+1}} }{t_a+t_{a+1}}  \\
& = \sum_{k=0}^n \frac{\binom{n}k\binom{an+1+k}{(a-1)n+1+k}}{an+1+k} \frac{t_a^{n-k}t_{a+1}^k}{t_a+t_{a+1}}.
\end{align*}

\noindent
As a first step, we justify that the linear term $t_a+t_{a+1}$ divides the polynomial
$$P_n(t_a,t_{a+1}):= \sum_{k=0}^n \frac{\binom{n}k\binom{an+1+k}{(a-1)n+1+k}}{an+1+k}  t_a^{n-k}t_{a+1}^k.$$
This is tantamount to $P_n(-t_{a+1},t_{a+1})=0$ which is equivalent to the identity that
$$ \sum_{k=0}^n (-1)^k \frac{\binom{n}k\binom{an+1+k}{(a-1)n+1+k}}{an+1+k} =0.$$
Again, apply the Wilf-Zeilberger approach with $F(n,k):= \frac{(-1)^k \binom{n}k\binom{an+1+k}{(a-1)n+1+k}}{an+1+k}$ and $H(n,k):=-F(n,k)\cdot \frac{k((a-1)n+1+k)}{n(an+1)}$ to verify $F(n,k)=H(n,k+1)-H(n,k)$. The rest is trivial.

\smallskip
\noindent
Our next step will actually determine $\widetilde{G}[m_a,m_{a+1}]$. To this effect, let's divide $\frac{P_n(t_a,t_{a+1})}{t_a+t_{a+1}}$
to obtain (routinely) that
\begin{align*}
[t_a^{n-1-i}t_{a+1}^i]\left(\frac{P_n(t_a,t_{a+1})}{t_a+t_{a+1}}\right)
& =\sum_{j=0}^i (-1)^{i-j}\frac{\binom{n}j\binom{an+1+j}{(a-1)n+1+j}}{an+1+j} \\
& = (-1)^i[H(n,i+1)-H(n,0)]  = (-1)^i H(n,i+1) \\
& = \frac1{an+1}\binom{n-1}i \binom{an+1+i}{(a-1)n+1+i}
\end{align*}
which proves the desired conjecture on $\widetilde{G}[m_a,m_{a+1}]$. To wit:

\begin{theorem} \label{conj2} Denote $m=m_a+m_{a+1}$. For integers $m_a, m_{a+1}\geq0$ there holds
$$\widetilde{G}[m_a,m_{a+1}]=\frac{(am_a+(a+1)(m_{a+1}+1))!}
{(a(m+1)+1)(m+1)((a-1)m_a+a(m_{a+1}+1))! m_a!m_{a+1}!}.$$
\end{theorem}

\section{On the third conjecture}

\noindent
The proof of the last conjecture \cite[p. 399]{WR} is a bit more complicated.

\begin{theorem} \label{conj3} For the $2a$-variate case, we have 
$$\pmb{G}[-f,f,\dots,-f,f]=\sum_n a^nf^n.$$
\end{theorem}

\begin{proof} To begin, we make a slight alteration by writing $(-1)^it_i$ instead of the customary plain $t_i$ \cite{WR}. Thanks to the Lagrange Inversion \eqref{lag}, we have
\begin{align*}
&[Y^n]\left(\sum_{k=1}^{\infty}\frac1k\, [z^{k-1}]\left(1-Yt_1z^2+Yt_2z^3-\dots-Yt_{2a-1}z^{2a}+Yt_{2a}z^{2a+1}\right)^k\right)  \\
= & [Y^n] \sum_{m_1,\dots,m_{2a}\geq0} 
\frac{ (-1)^{m_1+\cdots+m_{2a-1}}\binom{1+2m_1+3m_2+\cdots+(2a+1)m_{2a}}{m_1,m_2,\dots,m_{2a},1+m_1+2m_2+\cdots+(2a)m_{2a}}\,
(Yt_1)^{m_1}\cdots(Yt_{2a})^{m_{2a}}}
{1+2m_1+3m_2+\cdots+(2a+1)m_{2a}}   \\
= & \sum_{\substack{m_1,\dots,m_{2a}\geq0 \\ m_1+\cdots+m_{2a}=n}} 
\frac{ (-1)^{m_1+\cdots+m_{2a-1}}\binom{1+2m_1+3m_2+\cdots+(2a+1)m_{2a}}{m_1,m_2,\dots,m_{2a},1+m_1+2m_2+\cdots+(2a)m_{2a}} \,
t_1^{m_1}\cdots t_{2a}^{m_{2a}}}
{1+2m_1+3m_2+\cdots+(2a+1)m_{2a}}.
\end{align*}
First, consider the case $a=1$ and refer back to Theorem ~\ref{conj1} (and its proof), to gather that if $t_1=-f$ and $t_2=f$ then, as expected, we arrive at
$$f^{n-1}\sum_{m=0}^{n-1} \frac{(-1)^{n-1-m}}{2n+1}\binom{n-1}m \binom{2n+1+m}{n+1+m}=f^{n-1}$$
as justified by the \emph{WZ-certificate} \cite{WZ} given by
$$R(n,m):=\frac{m(8mn + 10n^2 + 6m + 15n + 6)}{2(2n + 3)(n + 1)(n - m)}.$$
Second, we go back to study the above-posed calculations when $a>1$. To set the stage, substitute $t_1=t_2=\cdots=t_{2a-1}=f$ while leaving out $t_{2a}$ as an indeterminate. The outcome takes the form
\begin{align*}
& \sum_{\substack{m_1,\dots,m_{2a}\geq0 \\ m_1+\cdots+m_{2a}=n}} 
\frac{ (-1)^{m_1+m_3+\cdots+m_{2a-1}}\binom{1+2m_1+3m_2+\cdots+(2a+1)m_{2a}}{m_1,m_2,\dots,m_{2a},1+m_1+2m_2+\cdots+(2a)m_{2a}} \,
f^{n-m_{2a}}  t_{2a}^{m_{2a}}}
{1+2m_1+3m_2+\cdots+(2a+1)m_{2a}}.
\end{align*}
At this point, divide out the current polynomial (in $t_{2a}$) by the linear factor
$$-t_1+t_2-\cdots-t_{2a-3}+t_{2a-1}-t_{2a-1}+t_{2a}=t_{2a}-f$$
and then replace $t_{2a}$ by $f$. That leads to the sum
\begin{align*}
 f^{n-1}  \sum_{i=0}^{n-1} \sum_{m_{2a}=0}^i
\sum_{\substack{m_1,\dots,m_{2a}\geq0 \\ m_1+\cdots+m_{2a}=n}} 
\frac{ (-1)^{1+m_1+m_3+\cdots+m_{2a-1}}\binom{1+2m_1+3m_2+\cdots+(2a+1)m_{2a}}{m_1,m_2,\dots,m_{2a},1+m_1+2m_2+\cdots+(2a)m_{2a}} }
{1+2m_1+3m_2+\cdots+(2a+1)m_{2a}  }.
\end{align*}
Therefore, our main task that remains is to prove the identity declared by
$$\sum_{i=0}^{n-1} 
\sum_{\substack{m_1,\dots,m_{2a-1}\geq0 \\ m_1+\cdots+m_{2a}=n \\ 0\leq m_{2a}\leq i} }
\frac{ (-1)^{1+m_1+m_3+\cdots+m_{2a-1}}\binom{1+2m_1+3m_2+\cdots+(2a+1)m_{2a}}{m_1,m_2,\dots,m_{2a},1+m_1+2m_2+\cdots+(2a)m_{2a}} }
{1+2m_1+3m_2+\cdots+(2a+1)m_{2a}  }.= a^{n-1}.$$

\noindent
To put this more succinctly, introduce some notation. Let $\mathcal{P}$ denote the set of all integer partitions $\lambda$, written as $\lambda=(\lambda_1,\lambda_2,\dots)$ or $\lambda=1^{m_1}2^{m_2}\dots (2a)^{m_{2a}}$.  The size of $\lambda$ is denoted 
by $\vert\lambda\vert=\lambda_1+\lambda_2+\cdots=m_1+2m_2+\cdots+(2a)m_{2a}$ while we use $\ell(\lambda)=m_1+m_2+\cdots+m_{2a}$ for the length of the partition. So, the claim stands at
\begin{align} \label{conj3claim}
\sum_{\substack{\lambda\in \mathcal{P} \\ \ell(\lambda)=n \\ \lambda_1\leq 2a}}  (-1)^{1+\vert\lambda\vert} \cdot
\frac{ (n-m_{2a})\, \binom{n}{m_1,\dots,m_{2a}} \binom{\vert\lambda\vert +n+1}{\vert\lambda\vert +1}}
{\vert\lambda\vert +n+1}= a^{n-1}.
\end{align}

\smallskip
\noindent
We find it more convenient to split up this assertion into two separate claims
\begin{align} \label{conj3claim*}
(-1)^1\sum_{\substack{\lambda\in \mathcal{P} \\ \ell(\lambda)=n \\ \lambda_1\leq 2a}}  (-1)^{\vert\lambda\vert}  \,
  \binom{n}{m_1,\dots,m_{2a}} \binom{\vert\lambda\vert+n}{\vert\lambda\vert +1} & = 0, \\   \label{conj3claim**}
\sum_{\substack{\mu\in \mathcal{P} \\ \ell(\mu)=n-1 \\ \mu_1\leq 2a}}  (-1)^{\vert\mu\vert} \,
\binom{n-1}{m_1,\dots,m_{2a}} \binom{\vert\mu\vert +2a+n}{\vert\mu\vert +2a+1} & = a^{n-1}.
\end{align}
One arrives at \eqref{conj3claim*} due to $\frac{n\,\binom{\vert\lambda+n+1}{\vert\lambda\vert+1}}{\vert\lambda+n+1}
=\binom{\vert\lambda+n}{\vert\lambda\vert+1}$ and \eqref{conj3claim**} arises because of $m_{2a}\binom{n}{m_1,\dots,m_{2a}}\frac{(\vert\lambda+n)!}{(\vert\lambda+1)!n!}=\binom{n-1}{m_1,\dots,m_{2a}-1}\binom{\vert\lambda\vert+n}{\vert\lambda\vert+1}$
and then we reindex $m_{2a}'=m_{2a}-1$ to convert $\vert\lambda\vert=\vert\mu\vert+2a$ where $\ell(\mu)=n-1$. 

\smallskip
\noindent
In fact,  let's generalize \eqref{conj3claim*} and \eqref{conj3claim**} by introducing an extra parameter $x$.

\smallskip
\noindent
\bf Claim 1: \rm For positive integers $n, a$ and an indeterminate $x$, we have
$$\sum_{\substack{\lambda\in \mathcal{P} \\ \ell(\lambda)=n \\ \lambda_1\leq 2a}}  (-1)^{\vert\lambda\vert}  \,
  \binom{n}{m_1,\dots,m_{2a}} \binom{\vert\lambda\vert+n+x}{n-1}=0.$$

\noindent
\bf Claim 2: \rm For positive integers $n, a$ and an indeterminate $x$, we have
$$\sum_{\substack{\lambda\in \mathcal{P} \\ \ell(\lambda)=n-1 \\ \lambda_1\leq 2a}}  (-1)^{\vert\lambda\vert}  \,
  \binom{n-1}{m_1,\dots,m_{2a}} \binom{\vert\lambda\vert+n+x}{n-1}=a^{n-1}.$$

\noindent
\it Claim 2 implies Claim 1: \rm We apply the multinomial recurrence (assume $n=k_1+\cdots+k_r$)
\begin{align} \label{multirec}
\binom{n}{k_1,\dots,k_r} & = \binom{n-1}{k_1-1,\dots,k_r}+\cdots+\binom{n-1}{k_1,\dots,k_r-1}
\end{align}
followed by appropriate reindexing so that 
\begin{align*}
& \sum_{\substack{\lambda\in \mathcal{P} \\ \ell(\lambda)=n \\ \lambda_1\leq 2a}}  (-1)^{\vert\lambda\vert}  \,
 \binom{n}{m_1,\dots,m_{2a}} \binom{\vert\lambda\vert+n+x}{n-1}  \\
= & \sum_{i=1}^{2a} \sum_{\substack{\lambda\in \mathcal{P} \\ \ell(\lambda)=n \\ \lambda_1\leq 2a}}  (-1)^{\vert\lambda\vert}  \,
\binom{n-1}{m_1,\dots,m_i-1,\dots m_{2a}} \binom{\vert\lambda\vert+n+x}{n-1}  \\
= &  \sum_{i=1}^{2a} \sum_{\substack{\mu\in \mathcal{P} \\ \ell(\mu)=n-1 \\ \mu_1\leq 2a}}  (-1)^{\vert\mu\vert+i}  \,
\binom{n-1}{m_1,\dots,m_i',\dots m_{2a}} \binom{\vert\mu\vert+n+(x+i)}{n-1}  \\
= & \sum_{i=1}^{2a} (-1)^i \sum_{\substack{\mu\in \mathcal{P} \\ \ell(\mu)=n-1 \\ \mu_1\leq 2a}}  (-1)^{\vert\mu\vert}  \,
\binom{n-1}{m_1,\dots,m_i',\dots m_{2a}} \binom{\vert\mu\vert+n+(x+i)}{n-1}  \\
= &  a^{n-1} \sum_{i=1}^{2a} (-1)^i  \\
= & \, 0.
\end{align*}

\noindent
\it Proof of Claim 2: \rm Let's now utilize the multinomial theorem \eqref{multi} and  constant-term extraction. Start by noting the constant-term extraction
\begin{align*} 
\binom{\vert\lambda\vert +n+ x}{n-1}   &=\binom{m_1+2m_2+\cdots+(2a)m_{2a}+n+x}{n-1}  \\
& =\pmb{CT}_z \left[\frac{(1+z)^{m_1+2m_2+\cdots+(2a)m_{2a}+n+x}}{z^{n-1}}\right].
\end{align*}
Insert this into the left-hand side of Claim 2, take $\pmb{CT}_z$ outside the sum, factor out the inside and reapply the multinomial theorem in reverse \eqref{multi} to get 
\begin{align*}
&  \sum_{\substack{\lambda\in \mathcal{P} \\ \ell(\lambda)=n-1 \\ \lambda_1\leq 2a}}  (-1)^{\vert\lambda\vert}  \,
    \binom{n-1}{m_1,\dots,m_{2a}} \binom{\vert\lambda\vert+n+x}{n-1}  \\
= & \pmb{CT}_z \left[ \frac{(1+z)^{n+x}}{z^{n-1}} \sum \binom{n-1}{m_1,\dots,m_{2a}} (-1-z)^{m_1}(-1-z)^{2m_2}\cdots(-1-z)^{(2a)m_{2a}} \right] \\
= & \pmb{CT}_z \left[ \frac{(1+z)^{n+x}}{z^{n-1}} \left\{ -(1+z)^1+(1+z)^2-(1+z)^3+\cdots+(1+z)^{2a}\right\}^{n-1} \right].
\end{align*}
Next, follow through with the geometric series expansion to obtain
\begin{align*}
&  \sum_{\substack{\lambda\in \mathcal{P} \\ \ell(\lambda)=n-1 \\ \lambda_1\leq 2a}}  (-1)^{\vert\lambda\vert}  \,
    \binom{n-1}{m_1,\dots,m_{2a}} \binom{\vert\lambda\vert+n+x}{n-1}  \\
= & \pmb{CT}_z \left[ (-1)^{n-1} \frac{(1+z)^{2n+x-1}}{z^{n-1}}\left\{\frac{1-(1+z)^{2a}}{2+z}\right\}^{n-1}  \right]   \\
= & \pmb{CT}_z \left[  \frac{(1+z)^{2n+x-1}}{(2z)^{n-1}}\left\{\frac{z\sum_{k=1}^{2a}\binom{2a}k z^{k-1}}{1+\frac{z}2}\right\}^{n-1}  \right]  =a^{n-1}.
\end{align*}
The proof is indeed complete.
\end{proof}

\begin{remark} On \cite[p. 399]{WR}, it is stated that \emph{``With $k-2$ leading zeros, we conjecture that $G[0,\dots,m_k]$ is a two-parameter Fuss-Catalan number."} In light of the conjectures we already proved, the current claim is rather obvious (for further discussion on the topic the reader is directed to \cite{M}).
\end{remark}

\begin{remark} One can prove both Theorem ~\ref{conj1} and \ref{conj2} with the following observation. It suffice to explain this for Theorem ~\ref{conj1}. Since $C[m_1,m_2]$ are known from the Lagrange Inversion and because we have and explicit conjectured formula $G[m_1,m_2]$ due to \cite{WR}, all that is required is to verify the relation $G[m_1-1,m_2]+G[m_1,m_2-1]=C[m_1,m_2]$. This, however, is routine. Of course, the proofs in Section s1 and 2 do not assume knowing $C[m_1,m_2]$ and $G[m_1,m_2]$ \emph{a priori}: they are pure \emph{derivations} from scratch. 
\end{remark}

\begin{remark}
We offer (the proof is analogous to Theorem ~\ref{conj2} but omitted) the assertion that
\begin{align*} 
G[0,\dots,0,m_s,0,\dots,m_t]
&= \frac1n\sum_{j=0}^i  (-1)^{i-j}\binom{n}j \binom{(s+1)n+(t-s)j}{n-1},
\end{align*}
where we used $m_s=n-1-i, m_t=i$.
\end{remark}

\begin{remark} We also offer (the proof is analogous to Theorem ~\ref{conj3} but omitted) the assertion that for a generalized $2a$-variate case, we have
\begin{align*}
& \pmb{G}[-c_af,c_1f,-c_1f,c_2f,-c_2f,\cdots,c_{a-1}f,-c_{a-1}f,c_af]  \\
= & \sum_n (2ac_a-c_1-c_2-\dots - c_{a})^nf^n.
\end{align*}
\end{remark}

\end{document}